\documentclass[12pt]{article}
\usepackage{amsmath,amsxtra,amssymb,latexsym, amscd,amsthm}

\advance\voffset-1truecm\relax
\advance\hoffset-1truecm\relax

\def\R{{\Bbb R}}
\def\N{{\Bbb N}}
\def\R{{\Bbb R}}
\def\C{{\Bbb C}}
\def\P{{\Bbb P}}

\numberwithin{equation}{section}

\newcommand {\Cal}{\mathcal}

\newtheorem{lemma}{Lemma}[section]

\newtheorem{corollary}[lemma]{Corollary}

\newtheorem{definition}[lemma]{Definition}
\newtheorem{remark}[lemma]{Remark}
\def\leq{\leqslant}

\begin{document}
\title{{Holomorphic curves into algebraic varieties intersecting moving hypersurface targets
%\\(Courbes holomorphes dans  des vari\'et\'es  alg\'ebriques intersectant des hyperpsurfaces mobiles)
}}
\author{Gerd Dethloff and Tran Van Tan}
\date{$\quad$}
\maketitle
%\vspace{-0.5cm}
\begin{abstract}
In [Ann. of Math.169 (2009)], Min Ru proved a second main theorem  for
algebraically nondegenerate holomorphic curves in complex projective varieties
 intersecting fixed hypersurface targets.
In this paper,  by introducing a new proof method for the case of projective varieties, we generalize this result  to moving hypersurface targets.
\\

%Dans  l'article [Ann. of Math.169 (2009)], Min Ru a obtenu un deuxi\`eme  th\'eor\`eme principal de la  th\'eorie de Nevanlinna dans le cas des courbes
%holomorphes alg\'ebriquement nond\'eg\'en\'er\'e  dans  des vari\'et\'es complexes projectives  intersectant des hypersurfaces  fix\'ees en position g\'en\'erale.  Dans cette article, en introduisant une nouvelle  m\'ethode de preuve pour le cas des  vari\'et\'es projectives, on g\'en\'eralise ce r\'esultat  aux  hypersurfaces mobiles.
%\\

\end{abstract}
\section{Introduction}
During the last century, several Second Main Theorems  have been established for linearly nondegenerate holomorphic curves in complex
projective spaces intersecting (fixed or moving) hyperplanes, and we now have a satisfactory  knowledge about it.
Motivated by a paper of Corvaja-Zannier \cite{CZ}
 in Diophantine approximation, in 2004 Ru \cite{R2} proved a Second Main Theorem
for algebraically nondegenerate holomorphic curves in the complex projective space $\C \P^n$ intersecting (fixed) hypersurface targets, which settled
a long-standing conjecture of Shiffman \cite{S}. In 2011, Dethloff-Tan \cite{DT} generalized this result of Ru to moving hypersurface targets
(this means where the coefficients of the hypersurfaces are meromorphic functions) in $\C \P^n.$ The counterpart of the Second Main Theorem of Dethloff-Tan in Diophantine approximation was independently given by  Chen-Ru-Yan \cite{CRY3} and Le \cite{G5}  in 2015.

In order to reduce the case of hypersurfaces  to the case of hyperplanes by an approximation, the above authors  construct a filtered vector space.
 The key point in their proof is the  following property: If homogeneous polynomials $Q_0,\dots, Q_n$ in $\C[x_0,\dots,x_n]$ have no  non-trivial common solutions, then $\{Q_0,\dots,Q_n\}$ is a regular sequence.
 Thanks to this property, they can construct  linear isomorphisms  showing that the dimension of all factor vector spaces in the filtration is exactly equal to the corresponding value of the Hilbert polynomial of a common algebraic variety, and then can be calculated. However, this regular sequence property is not true any more for the general case of varieties $V \subset \C \P^M$,
 and is related to whether or not the homogenous coordinate ring of $V$ is Cohen-Macauley, moreover,  there are examples to show that their linear isomorphisms can not be extended to  the general case of varieties. So by dropping this restriction on the variety $V$ and thereby losing regular sequences, we can no longer exactly calculate the dimensions of the various factor  vector spaces of the filtration by using this method.

 The result of Corvaja-Zennier \cite{CZ} was actually reproved and generalized to the case of  projective varieties by Evertse-Ferretti \cite{EF2} using somewhat different ideas. Their approach also admitted extensions
in Nevanlinna theory which were developped again by Ru \cite{R3} in 2009,
obtaining  a Second Main Theorem for entire curves in arbitrary projective varieties $V,$ with respect to
hypersurfaces of the ambient space. To prove the Second Main Theorem, in \cite{R3}, Ru uses the finite morphism $\phi : V \to \C \P^{q-1},$
$\phi(x) := [Q_1(x) :\cdots : Q_q(x)],$  where the $Q_j$'s are homogeneous polynomials (with common degree) defining
the given hypersurfaces. Thanks to this finite morphism, he can use a generalization of  Mumford's identity (the version with explicit estimates obtained by
 Evertse and Ferretti \cite{EF, EF2})
 for the variety Im$\phi \subset\C \P^{q-1}.$ However, for the case of moving hypersurfaces,
 we do not have such a morphism.

 The purpose of this article is twofold, giving a  Second Main Theorem for entire curves in projective varieties intersecting moving hypersurfaces and introducing  a new approach for the case of projective varieties.
 Firstly,  we show that by specializing the coefficients of the polynomials corresponding to the moving hypersufaces in generic points,  the dimensions of the given vector spaces do not change.
Secondly, we  construct a  filtration of the vector space corresponding to the coordinate ring of the variety. After that, by observing the Hilbert sequence asymptotics,
we calculate the sum of the dimensions of all the factors of the vector spaces in the filtration and by using the algebraic properties of our filtration, properties of its Hilbert function and also techniques in combinatorics, we prove that almost all of these factor vector spaces have the same dimension.
Finally, we prove that  we can neglect the other factors vector spaces of the filtration where the dimension is not as expected.
 Another difficulty in the case of moving hypersurface targets is that they are in general position only for generic points. In order to overcome this difficulty, we use an element of the ideal of the inertia forms of the system of polynomials defining the moving hypersurfaces in order to control the locus where the divisors are not in general position. In fact, the ideal of  the inertia forms of such a system of polynomials  is not a principal ideal in general
 (unless $V$ is a complete intersection variety). But we will show that there  exists an element of this ideal with properties which are enough for our purpose.

Our  method was used again  by Ji-Yan-Yu \cite{JYY}, Yan-Yu \cite{YY}, Son-Tan-Thin \cite{STT} to prove  their   Second Main Theorems and  Schmidt's Subspace Theorems.

Let $f$ be a holomorphic mapping of $\C$ into $\C \P^M,$ with a reduced representation $f:=(f_0:\cdots: f_M).$ The characteristic function $T_f(r)$ of
$f$ is defined by
\begin{align*}
T_f(r):=\frac{1}{2\pi}\int\limits_0^{2\pi}\log\Vert f(re^{i\theta})\Vert d\theta,\quad r>1,
\end{align*}
where $\Vert f\Vert:= \max\{|f_0|,\dots,|f_M|\}.$

Let $\nu$ be a divisor on $\C.$ The counting function of $\nu$ is defined by
\begin{align*}
N_{\nu}(r):=\int\limits_1^{r}\log\frac{\sum_{|z|<t}\nu(z)}{t}dt,\quad r>1.
\end{align*}
For a non-zero meromorphic function $\varphi,$ denote by $\nu_\varphi$ the zero divisor
 of $\varphi,$ and set $N_\varphi(r):=N_{\nu_\varphi}(r).$
Let $Q$ be a homogeneous polynomial in the variables $x_0,\dots,x_M$ with coefficients which are meromorphic functions.
If $Q(f):=Q(f_0,\dots,f_M)\not\equiv 0$,
we define $N_f(r,Q):=N_{Q(f)}(r).$ Denote by $Q(z)$ the homogeneous  polynomial over $\C$ obtained by
evaluating the coefficients of $Q$ at a specific point $z \in \C$
in which all coefficient functions of $Q$ are holomorphic (in particular $Q(z)$ can be the zero polynomial).

We say that a meromorphic function $\varphi$ on $\C$ is ``small"
with respect to $f$ if $T_\varphi(r) = o(T_f(r))$ as $r \to \infty$
(outside a set of finite Lebesgue measure).
%\newpage

Denote by $\Cal K_f$ the set of all ``small" (with respect to $f$)
meromorphic functions on $\C$. Then $\Cal K_f$ is a field.

For a positive integer $d$, we set
\begin{align*}
\Cal T_d := \big\{ (i_0,\dots,i_M) \in \N_0^{M+1} :
i_0 + \dots + i_M = d \big\}.
\end{align*}
Let $\Cal Q=\{Q_1,\dots,Q_q\}$ be a set of $q\geq n+1$ homogeneous polynomials in $\Cal K_f[x_0,\dots,x_M],$
$\text{deg}\,Q_j = d_j \geq 1.$ We write
\begin{align*}
Q_j = \sum\limits_{I \in \Cal T_{d_j}} a_{jI}x^I \quad (j = 1,\dots,q)
\end{align*}
 where $x^I = x_0^{i_0} \cdots x_M^{i_M}$
for $x = (x_0,\dots,x_M)$ and $I = (i_0, \dots,i_M)$.
Denote by $\Cal K_{\Cal Q}$ the field over $\C$ of   all
meromorphic functions on $\C$ generated by
$\big\{ a_{jI} : I \in \Cal T_{d_j}, j \in \{1,\dots,q\}\big\}$.
It is clearly a subfield of $\Cal K_f$.

Let $V\subset\C \P^M$ be an arbitrary projective variety of dimension $n$, generated by the homogeneous polynomials in its ideal $\Cal I(V)$.
Assume that $f$ is   non-constant and Im$f\subset V.$ Denote by $\Cal I_{\Cal K_{\Cal Q}}(V)$ the ideal
in $\Cal K_{\Cal Q}[x_0,\dots,x_M]$ generated by  $\Cal I(V)$. Equivalently $\Cal I_{\Cal K_\Cal Q}(V)$ is
 the (infinite-dimensional) $\Cal K_\Cal Q$-sub-vector space of $\Cal K_{\Cal Q}[x_0,\dots,x_M]$ generated by  $\Cal I(V)$. We note that $Q(f)\equiv 0$ for every homogeneous polynomial
$Q\in\Cal I_{\Cal K_\Cal Q}(V).$
We say that $f$ is algebraically nondegenerate over
$\Cal K_{\Cal Q}$
if there is no  homogeneous polynomial
$Q \in \Cal K_{\Cal Q}[x_0,\dots,x_M]\setminus\Cal I_{\Cal K_\Cal Q}(V)$
such that
$Q(f) \equiv 0.$

The set $\Cal Q$ is said to be $V-$ admissible (or in (weakly) general position (with respect to $V$))
if there exists $z \in \C$
in which all coefficient functions of all $Q_j$, $j=1,...,q$ are holomorphic and such that for any
$1 \leq j_0 < \dots < j_n \leq q$ the system of equations
\begin{align} \label{zz}
\left\{ \begin{matrix}
Q_{j_i}(z)(x_0,\dots,x_M) = 0\cr
0 \leq i \leq n\end{matrix}\right.
\end{align}
has no solution  $(x_0, \dots , x_M)$ satisfying $(x_0:\cdots:x_M)\in V.$
As we will show in section 2, in this case this is true for all $z \in \C$ excluding a discrete subset of $\C$.

As usual, by the notation ``$\Vert P$" we mean that the assertion $P$ holds
for all $r \in [1, +\infty)$ excluding a Borel subset $E$ of $(1, +\infty)$
with $\displaystyle{\int\limits_E} dr < +  \infty$.

Our main result is stated as follows:

\noindent{\bf Main Theorem}. {\it Let $V\subset \C \P^M$ be an irreducible (possibly singular) variety of dimension $n,$ and let $f$ be a non-constant
holomorphic map of $\C$ into $V.$ Let
$\Cal Q=\{Q_1,\dots,Q_q\}$ be a $V-$ admissible set of homogeneous
polynomials in $\Cal K_f [x_0,\dots,x_M]$ with $\deg Q_j = d_j \geq 1$. Assume that $f$ is algebraically nondegenerate over $
\Cal K_{\Cal Q}$. Then for any $\varepsilon > 0,$
\begin{align*}
\Vert (q-n-1-\varepsilon) T_f(r) \leq \sum_{j=1}^q \frac{1}{d_j}
N_f(r,Q_j).
\end{align*}}
In the special case where the coefficients of the polynomials $Q_j$'s are constant and the variety $V$ is smooth,
the above theorem is the Second Main Theorem of Ru in \cite{R3}. According to Vojta (\cite{VS}, p. 183)
generalizing this theorem of Ru to singular varieties can be done already by his proof methods without essential changes of the proof
(see also Chen-Ru-Yan \cite{CRY1}, \cite{CRY2}), so
the essential generalization in our main result is the one to moving targets.

We define the defect of $f$ with respect to a homogenous polynomial $Q\in\Cal K_f[x_0,\dots,x_M]$
of degree $d$ with $Q(f)\not\equiv 0$
by
\begin{align*}
\delta_f(Q):=\lim_{r\to+\infty}\inf\Big(1-\frac{N_f(r,Q)}{d\cdot T_f(r)}\Big).
\end{align*}
As a corollary of the Main Theorem we get the following defect relation.
\begin{corollary} Under the assumptions  of the Main theorem,
we have
\begin{align*}
\sum_{j=1}^q \delta_f (Q_j) \leq n+1.
\end{align*}
\end{corollary}
\noindent {\bf Acknowledgements:} This research was supported by the
Vietnam National Foundation for Science and Technology
Development (NAFOSTED) under grant number 101.02-2016.17. The authors were partially supported by the Vietnam Institute for Advanced Studies in Mathematics. The second named author was partially supported by the   Institut des Hautes \'Etudes Scientifiques (France),
 and by a travel grant from the Simons Foundation.
He also would like to thank William Cherry, Christophe  Soul\'e,  Ofer Gabber and  Laurent Buse for valuable discussions.  Finally both authors would like to thank Min Ru and Gordon Heier for pointing out to them an error in the proof of the first version of this paper (more precisely the second part of equation (2.12) in arXiv 1503.08801v1 (30 MAR 2015) does not hold in general). We have  corrected this error in arXiv 1503.08801v2 (13 FEB 2017).

\section{Lemmas}
Let $\Cal K$ be an arbitrary field over $\C$ generated by a set of meromorphic functions on $\C.$ Let $V$ be a sub-variety in $\C \P^M$ of dimension $n$
defined by the homogeneous ideal $\Cal I(V) \subset \C[x_0,\dots,x_M].$ Denote by $\Cal I_{\Cal K}(V)$ the ideal in $\Cal K[x_0,\dots,x_M]$
 generated by $\Cal I(V).$

For each positive integer $k$ and for any (finite or infinite dimensional) $\C$-vector sub-space $W$ in $\C[x_0,\dots,x_M]$ or
for any  $\cal K$-vector sub-space $W$ in  $\Cal K[x_0,\dots,x_M]$, we denote by $W_k$
the vector sub-space consisting of all
homogeneous polynomials in $W$ of degree $k$ (and of the zero polynomial; we remark that
$W_k$ is necessarily of finite dimension).

\noindent The Hilbert polynomial $H_V$  of $V$ is defined by
 \begin{align*}
H_V(N):=\dim_{\C}\frac{\C[x_0,\dots,x_M]_N}{\Cal I(V)_N},\quad N\in\N_0.
\end{align*}
By the usual theory  of Hilbert polynomials (see e.g. \cite{H}), for $N>>0,$ we have
\begin{align*}H_V(N)
=\deg V\cdot \frac{N^n}{n!} +O(N^{n-1}).
\end{align*}
\begin{definition}
Let $W$ be a $\cal K$-vector sub-space in $\Cal K[x_0,\dots,x_M].$  For each $z\in \C,$  we denote
$$W(z):=\{P(z): P\in W,\; \text{all coefficients of}\; P \;\text{are holomorphic at}\; z\}.$$
\end{definition}
 It is clear that $W(z)$ is a $\C$-vector sub-space of $\C[x_0,\dots,x_M].$
\begin{lemma} \label{L0}
Let $W$ be a $\cal K$-vector sub-space in $\Cal K[x_0,\dots,x_M]_N.$  Assume that $\{h_j\}_{j=1}^K$ is a basis of $W$. Then $\{h_j(a)\}_{j=1}^K$ is a basis of $W(a)$
  (and in particular $\dim_{\Cal K} W=\dim_{\C}W(a)$) for
all $a\in\C$ excluding a discrete subset.
\end{lemma}
\begin{proof}
 Let $(c_{ij})$ be the matrix of coefficients of $\{h_j\}_{j=1}^K.$ Since
$\{h_j\}_{j=1}^K$ are linearly independent over $\cal K$, there exists a square submatrix $A$ of $(c_{ij})$ of order $K$ and such that $\det A\not\equiv 0.$
Let  $a$ be an arbitrary point in $\C$ such that $\det A(a)\ne 0$
and such that all coefficients of $\{h_j\}_{j=1}^K$ are holomorphic  at $a.$ For each $P\in W$ whose coefficients are all holomorphic at $a,$
we write $P=\sum_{j=1}^Kt_jh_j$ with $t_j\in\Cal K.$ In fact, there are coefficients $b_j$ $(j=1,\dots,K)$ of $P$ such that $(t_1,\dots,t_K)$ is the unique
solution in $\Cal K^K$ of the following system of linear equations:
\begin{align*}
A\cdot \left ( \begin{matrix}
t_1\cr
\cdot\cr
\cdot\cr
\cdot\cr
t_K\end{matrix}\right)=\left ( \begin{matrix}
b_1\cr
\cdot\cr
\cdot\cr
\cdot\cr
b_K\end{matrix}\right).
\end{align*}
 By our choice of $a,$ so in particular we have $\det A(a)\ne 0$, and since  $\{b_j\}_{j=1}^K$ are holomorphic at $a$,
we get that the $\{t_j\}_{j=1}^K$ are holomorphic at $a.$
Therefore, $P(a)=\sum_{j=1}^Kt_j(a)h_j(a),$ $t_j(a)\in\C.$ On the other hand,  still by our choice of $a$, we have $h_j(a)\in W(a)$ for all $j\in\{1,\dots,K\}.$ Hence,
 $\{h_j(a)\}_{j=1}^K$ is a generating system of $W(a).$
Since $\det A(a)\ne 0,$  the matrix $(c_{ij}(a))$ has maximum rank. Therefore, $\{h_j(a)\}_{j=1}^K$ are also linearly independent over $\C$.
\end{proof}

Throughout of this section, we consider a $V-$ admissible set of $(n+1)$  homogeneous polynomials $Q_0,\dots,Q_n$
in $\Cal K[x_0,\dots,x_M]$ of common degree $d.$
We write
\begin{align*}
Q_j = \sum\limits_{I \in \Cal T_{d}} a_{jI}x^I, \quad (j = 0,\dots,n),
\end{align*}
where $a_{jI}\in\Cal K$ and $\Cal T_d$ is again the set of all $I:=(i_0,\dots,i_M)\in \N_{0}^{M+1}$ with $i_0+\cdots+i_M=d$.

\noindent Let $ t=(\dots, t_{jI},\dots)$ be a family of variables. Set
\begin{align*}
\widetilde{Q_j} = \sum\limits_{I \in \Cal T_{d}} t_{jI}x^I\in\C[t,x], \quad (j = 0,\dots,n).
\end{align*}
We have
$$\widetilde{Q_j}(\dots,a_{jI}(z),\dots,x_0,\dots, x_M)=Q_j(z)(x_0,\dots,x_M).$$
Assume that the ideal $\Cal I(V)$ is generated by homogeneous polynomials $P_1,\dots,P_m.$ Since $\{Q_0,\dots, Q_n\}$ is a $V-$ admissible set,
there exists $z_0\in\C$ such that
the homogeneous polynomials $P_1,\dots,P_m,Q_0(z_0),\dots,Q_n(z_0)$ in $\C[x_0,\dots,x_M]$ have no common non-trivial solutions.
Denote by $_{\C[t]}(P_1,\dots, P_m, \widetilde{Q_0},\dots,\widetilde{Q_n})$
 the ideal in the ring of  polynomials in $x_0,\dots, x_M$ with coefficients in $\C[t]$ generated by
$P_1,\dots,P_m,\widetilde{Q_0},\dots,\widetilde{Q_n}.$
 A polynomial $\widetilde R$ in $\C[t]$  is called an
 {\it inertia form} of the polynomials $P_1,\dots, P_m, \widetilde{Q_0},\dots,\widetilde{Q_n}$ if it has the following property  (see e.g.  \cite{Z}):
\begin{align*}
x_i^s\cdot \widetilde {R}\in \;_{\C[t]}(P_1,\dots, P_m, \widetilde {Q_0},\dots,\widetilde{Q_n})
\end{align*}
for $i=0,\dots,M$ and for some non-negative integer $s.$

It is well known that for the $(m+n+1)$ homogeneous polynomials $P_i(x_0,\dots,x_M),$ $\widetilde{Q_j}(\dots,t_{jI},\dots,x_0,\dots, x_M),$ $ i\in\{1,\dots,m\},$
$j\in\{0,\dots,n\}$
there exist finitely many inertia forms
$\widetilde R_1,..., \widetilde R_s$ (which are homogeneous polynomials in the $t_{jI}$ separately for each
$j \quad (j = 0,\dots,n)$ ) such that the following holds :
For special values $t_{jI}^0$ of $t_{jI}$
the $(m+n+1)$ homogeneous polynomials $P_i(x_0,\dots,x_M),$ $\widetilde{Q_j}(\dots,t_{jI}^0,\dots,x_0,\dots, x_M),$ $ i\in\{1,\dots,m\},$
$j\in\{0,\dots,n\}$
have a
common non-trivial solution in $x_0,\dots,x_M$
if and only if
 $t_{jI}^0$  is a common zero of the inertia forms $\widetilde R_1,..., \widetilde R_s$
(see e.g. \cite{H}, page 35 or \cite{Z}, page 254).
Choose such a $\widetilde R$ for the special values $t_{jI}^0= a_{jI}(z_0)$, and put
$R(z):=\widetilde R(\dots,a_{kI}(z),\dots) \in \Cal K$.
Then by construction, $R(z_0)\ne 0,$ hence $R\in \Cal K\setminus \{0\}$, so in particular $R$
only vanishes on a discrete subset of $\C$, and, by the above property of the inertia form
$\widetilde R$, outside this discrete subset,
 $Q_0(z),\dots,Q_n(z)$
have no common solutions in $V$.
 Furthermore, by the definition of the inertia forms, there exists a non-negative integer $s$ such that
\begin{align}\label{2.1}
  x_i^s\cdot  R\in \;_{\Cal K}(P_1,\dots,P_m, Q_0,\dots,Q_n), \; \text{for} \; i=0,\dots,M,
\end{align}
 where $\;_{\Cal K}(P_1,\dots,P_m, Q_0,\dots,Q_n)$ is the ideal in $\Cal K[x_0,\dots,x_M]$
 generated by
$P_1,\dots,P_m, Q_0,\dots,Q_n.$

Let $f$ be a nonconstant meromorphic map of $\C$ into $\C \P^M$.
Denote by $\Cal C_f$ the set of
all    non-negative functions
$h : \C  \longrightarrow [0,+\infty] \subset \overline{\R}$,
which are of the form
\begin{equation}\label{expr}
\frac{|u_1|+ \dots + |u_k|}{|v_1|+ \dots + |v_\ell|}\: ,
\end{equation}
where $k,\ell \in \N$, $u_i, v_j \in \Cal K_f \setminus \{0\}.$ \\
By the First Main Theorem we have
\begin{align*}
\frac{1}{2\pi}\int\limits_0^{2\pi} \text{log}^+ |\phi (re^{i\theta})| d\theta = o(T_f(r)),\quad
\text{as}\ r \to \infty
\end{align*}
for $\phi \in \Cal K_f $.
Hence,  for any $h \in \Cal C_f$, we have
\begin{align*}
\frac{1}{2\pi}\int\limits_0^{2\pi} \text{log}^+ h(re^{i\theta}) d\theta= o(T_f(r)),\quad
\text{as}\ r \to \infty.
\end{align*}
It is easy to see that sums, products and quotients of functions
in $\Cal C_f$ are again in $\Cal C_f$.

By the result on the inertia forms mentioned above, similarly to Lemma 2.2  in \cite{DT}, we have
\begin{lemma}\label{inertia}
Let $\big\{Q_j\big\}_{j=0}^n$ be a $V-$ admissible set of homogeneous
polynomials  of degree $d$ in $\Cal K [x_0,\dots,x_M]$. If $\Cal K\subset\Cal K_f,$ then there exist
functions $h_1,h_2 \in \Cal C_f \setminus \{ 0 \}$ such that,
\begin{align*}
h_2 \cdot \Vert f \Vert^d\leq\max_{j \in \{0,\dots,n\}}
|Q_j(f_0,\dots,f_M)| \leq h_1 \cdot \Vert f \Vert^d.
\end{align*}
\end{lemma}
\noindent In fact, the second inequality is elementary. In order to obtain the first inequality, we use equation
(\ref{2.1}) in the same way as the corresponding equation in Lemma 2.1 in \cite{DT}, and we observe
that we have $P_i(f_0,...,f_M) \equiv 0$ for $i=1,...,m$ since $f(\C) \subset V$, so the maximum
only needs to be taken over the $Q_j(f_0,...,f_M)$, $j=0,...,n$. The rest of the proof is identically
to the one of Lemma 2.2 in \cite{DT}.

 We use the lexicographic order in $\N_0^n$ and for $I=(i_1,\dots,i_n),$ set $\Vert I\Vert :=i_1+\cdots+i_n.$
\begin{definition} For each
 $I=(i_1,\cdots, i_n)\in \N_0^n$ and $N\in\N_0$ with $N\geq d\Vert I\Vert,$ denote by $\Cal L_N^I$
the set of all $\gamma\in\Cal K[x_0,\dots,x_M]_{N-d\Vert I\Vert}$ such that
%for each  $E>I$ there
%exists $\gamma_E\in \Cal K [x_0,\dots,x_M]_{N-d\Vert E\Vert}$ satisfying
\begin{align*}
Q_1^{i_1}\cdots Q_n^{i_n}\gamma-
\sum_{E=(e_1,\dots,e_n)>I}Q_1^{e_1}\cdots Q_n^{e_n}\gamma_E\in \mathcal I_{\Cal K}(V)_N.
\end{align*}
for Êsome $\gamma_E\in \Cal K [x_0,\dots,x_M]_{N-d\Vert E\Vert}$.
\end{definition}
Denote by $\Cal L^I$ the homogeneous ideal in $\Cal K [x_0,\dots,x_M]$ generated by $\cup_{N\geq d\Vert I\Vert}\Cal L_N^I.$
\begin{remark}\label{r1} i) $\Cal L_N^I$ is a $\Cal K$-vector sub-space of $\Cal K[x_0,\dots,x_M]_{N-d\Vert I\Vert},$ and
 $(\Cal I(V), Q_1,\dots,Q_n)_{N-d\Vert I\Vert}\subset \Cal L_N^I,$   where $(\Cal I(V), Q_1,\dots,Q_n)$ is
 the ideal in $\Cal K[x_0,\dots,x_M]$ generated by
 $\Cal I(V)\cup\{Q_1,\dots,Q_n\}.$

ii) For any $\gamma\in \Cal L_N^I$ and $P\in \Cal K[x_0,\dots,x_M]_k,$ we have $\gamma\cdot P\in\Cal L_{N+k}^I$

iii) $\Cal L^I\cap \Cal K[x_0,\dots,x_M]_{N-d\Vert I\Vert}=\Cal L_N^I.$

iv) $\frac{\Cal K[x_0, \dots, x_M]}{\Cal L^I}$
is a graded modul over the graded ring $\Cal K[x_0,\dots,x_M].$
\end{remark}
Set \begin{align*}
m_N^I:=\dim_{\Cal K}\frac{\Cal K[x_0,\dots,x_M]_{N-d\Vert I\Vert}}{\Cal L_N^I}.
%,\;\text{and}\; n_N^I:=\dim_{\Cal K}\frac{\Cal L_N^I}{(\Cal I(V), Q_1,\dots,Q_n)_{N-d\Vert I\Vert}}.
\end{align*}
For each positive integer $N,$ denote by $\tau_N$ the set of all $I:=(i_0,\dots,i_n)\in \N_0^n$ with $N-d\Vert I\Vert\geq 0.$
Let $\gamma_{I1},\dots,\gamma_{Im_N^I}\in \Cal K[x_0,\dots,x_M]_{N-d\Vert I\Vert}$ such that they form a basis of the $\Cal K$-vector space $\frac{\Cal K[x_0,\dots,x_M]_{N-d\Vert I\Vert}}{{\Cal L_N^I}}.$
\begin{lemma}\label{L1}
$\{[Q_1^{i_1}\cdots Q_n^{i_n}\cdot\gamma_{I1}],\dots,[Q_1^{i_1}\cdots Q_n^{i_n}\cdot\gamma_{Im_N^I}],\;I=(i_1,\dots,i_n)\in\tau_N\}$
is a basis of the $\Cal K$-vector space
$\frac{\mathcal K[x_0,\dots,x_M]_N}{\mathcal I_{\Cal K}(V)_N}.$
\end{lemma}
\begin{proof}
Firstly, we prove that:
\begin{align}\label{a0}
\{[Q_1^{i_1}\cdots Q_n^{i_n}\cdot\gamma_{I1}],\dots,[Q_1^{i_1}\cdots Q_n^{i_n}\cdot\gamma_{Im_N^I}],\;I=(i_1,\dots,i_n)\in\tau_N\}
\end{align}
 are linearly independent.

\noindent Indeed,
let $t_{I\ell}\in\Cal K,$ $(I=(i_1,\dots,i_n)\in\tau_N, \ell\in\{1,\dots,m_N^I\})$ such that
\begin{align*}
\sum_{I\in\tau_N}\big(t_{I1}[Q_1^{i_1}\cdots Q_n^{i_n}\cdot\gamma_{I1}]+\cdots+t_{Im_N^I}[Q_1^{i_1}\cdots Q_n^{i_n}\cdot\gamma_{Im_N^I}]\big)=0.
\end{align*}
Then
\begin{align}\label{a1}
\sum_{I\in\tau_N}Q_1^{i_1}\cdots Q_n^{i_n}\big(t_{I1}\gamma_{I1}+
\cdots+ t_{Im_N^I}\gamma_{Im_N^I}\big)\in\mathcal I_{\Cal K}(V)_N.
\end{align}
By the definition of $\Cal L_N^I,$ and by (\ref{a1}), we get
\begin{align*}
t_{I^*1}\gamma_{I^*1}+
\cdots+ t_{I^*m_N^{I^*}}\gamma_{I^*m_N^{I^*}}\in\mathcal L_N^{I^*},
\end{align*}
where $I^*$ is the smallest element of $\tau_N.$\\
On the other hand, $\{\gamma_{I^*1},\dots,\gamma_{I^*m_N^{I^*}}\}$ form a basis of $\frac{\Cal K[x_0,\dots,x_M]_{N-d\Vert I^*\Vert}}{\mathcal L_N^{I^*}}.$
 \\
Hence,
\begin{align}\label{a2}
t_{I^*1}=\cdots=t_{I^*m_N^{I^*}}=0.
\end{align}
Then, by (\ref{a1}), we have
\begin{align*}
\sum_{I\in\tau_N\setminus \{I^*\}}Q_1^{i_1}\cdots Q_n^{i_n}\big(t_{I1}\gamma_{I1}+
\cdots+ t_{Im_N^I}\gamma_{Im_N^I}\big)\in\mathcal I_{\Cal K}(V)_N.
\end{align*}
Then, similarly to (\ref{a2}), we have
\begin{align*}
t_{\tilde I1}=\cdots=t_{\tilde Im_N^{\tilde I}}=0,
\end{align*}
where $\tilde I$ is the smallest element of $\tau_N\setminus\{I^*\}.$

\noindent Continuing the above process, we get that $t_{I\ell}=0$ for all $I\in\tau_N$ and $\ell\in\{1,\dots, m_N^I\},$ and hence, we get (\ref{a0}).

Denote by $\Cal L$ the $\Cal K$-vector sub-space in $\Cal K[x_0,\dots,x_M]_N$ generated by
$$\{Q_1^{i_1}\cdots Q_n^{i_n}\cdot\gamma_{I1},\dots,Q_1^{i_1}\cdots Q_n^{i_n}\cdot\gamma_{Im_N^I},\;I=(i_1,\dots,i_n)\in\tau_N\}.$$

Now we prove that: For any $I=(i_1,\dots,i_n)\in\tau_N,$  we have
\begin{align}\label{a3}
Q_1^{i_1}\cdots Q_n^{i_n}\cdot\gamma_I\in \Cal L+\mathcal I_{\Cal K}(V)_N
\end{align}
for all $\gamma_I\in\Cal K[x_0,\dots,x_M]_{N-d\Vert I\Vert}.$

Set $I'=(i'_1,\dots,i'_n):=\max\{I: I\in\tau_N\}.$
Since $\gamma_{I'1},\dots,\gamma_{I'm_N^{I'}}$
form a basis of $\frac{\Cal K[x_0,\dots,x_M]_{N-d\Vert I'\Vert}}{{\Cal L_N^{I'}}},$  for any  $\gamma_{I'}\in\Cal K[x_0,\dots,x_M]_{N-d\Vert I'\Vert},$ we have
\begin{align}\label{a4}
\gamma_{I'}=\sum_{\ell=1}^{m_N^{I'}}t_{I'\ell}\cdot\gamma_{I'\ell}+h_{I'\ell},\; \text{where} \;h_{I'\ell}\in\Cal L_N^{I'},\;\text{and}\;t_{I'\ell}
\in\Cal K.
\end{align}
On the other hand, by the definition of $\Cal L_N^{I'}$ , we have
 $Q_1^{i'_1}\cdots Q_n^{i'_n}\cdot h_{I'\ell}\in \mathcal I_{\Cal K}(V)_N$
(note that $I'=\max\{I: I\in\tau_N\}).$
\noindent Hence,
\begin{align*}
Q_1^{i'_1}\cdots Q_n^{i'_n}\cdot\gamma_{I'}=\sum_{\ell=1}^{m_N^{I'}}t_{I'\ell}Q_1^{i'_1}\cdots Q_n^{i'_n}
\cdot\gamma_{I'\ell}+Q_1^{i'_1}\cdots Q_n^{i'_n}\cdot h_{I'\ell}
\in \Cal L+\mathcal I_{\Cal K}(V)_N.
\end{align*}
We get (\ref{a3}) for the case where $I=I'.$

Assume that (\ref{a3}) holds for all $I>I^*=(i^*_1,\dots,i^*_n).$ We prove that (\ref{a3}) holds also for $I=I^*.$

\noindent Indeed, similarly to (\ref{a4}), for any  $\gamma_{I^*}\in\Cal K[x_0,\dots,x_M]_{N-d\Vert I^*\Vert},$ we have
\begin{align*}
\gamma_{I^*}=\sum_{\ell=1}^{m_N^{I^*}}t_{I^*\ell}\cdot\gamma_{I^*\ell}+h_{I^*\ell},\; \text{where} \;h_{I^*\ell}\in\Cal L_N^{I^*},
\;\text{and}\;t_{I^*\ell}\in\Cal K.
\end{align*}
Then,
\begin{align}\label{a5}
Q_1^{i^*_1}\cdots Q_n^{i^*_n}\cdot\gamma_{I^*}=\sum_{\ell=1}^{m_N^{I^*}}t_{I^*\ell}Q_1^{i^*_1}\cdots Q_n^{i^*_n}
\cdot\gamma_{I^*\ell}+Q_1^{i^*_1}\cdots Q_n^{i^*_n}\cdot h_{I^*\ell}.
\end{align}
Since $h_{I^*\ell}\in\Cal L_N^{I^*},$ we have
\begin{align*}
Q_1^{i^*_1}\cdots Q_n^{i^*_n}\cdot h_{I^*\ell} - \sum_{E=(e_1,\dots,e_n)>I^*}Q_1^{e_1}\cdots Q_n^{e_n}\cdot g_{E}\in \mathcal I_{\Cal K}(V)_N,
\end{align*}
for some $g_E\in\Cal K[x_0,\dots,x_M]_{N-d\Vert E\Vert}.$

\noindent Therefore,  by the induction hypothesis,
\begin{align*}
Q_1^{i^*_1}\cdots Q_n^{i^*_n}\cdot h_{I^*\ell}\in \mathcal L+\mathcal I_{\Cal K}(V)_N.
\end{align*}
Then, by (\ref{a5}), we have
\begin{align*}
Q_1^{i^*_1}\cdots Q_n^{i^*_n}\cdot\gamma_{I^*}\in\mathcal L+\mathcal I_{\Cal K}(V)_N.
\end{align*}
This means that (\ref{a3}) holds for $I=I^*.$ Hence, by  (descending) induction we get (\ref{a3}).

For any $Q\in\Cal K[x_0,\dots,x_M]_N,$ we write $Q=Q_1^0\cdots Q_n^0\cdot Q.$
Then by (\ref{a3}), we have
$$Q\in\mathcal L+\mathcal I_{\Cal K}(V)_N.$$
Hence,
\begin{align*}
\{[Q_1^{i_1}\cdots Q_n^{i_n}\cdot\gamma_{I1}],\dots,[Q_1^{i_1}\cdots Q_n^{i_n}\cdot\gamma_{Im_N^I}],\;I=(i_1,\dots,i_n)\in\tau_N\}
\end{align*}
is a generating system of $\frac{\mathcal K[x_0,\dots,x_M]_N}{\mathcal I_{\Cal K}(V)_N}.$
Combining with (\ref{a0}), we get the conclusion of Lemma \ref{L1}.
\end{proof}

\begin{lemma} \label{inftylemma} $\#\{\Cal L^I: I\in \N_0^n\}<\infty.$
\end{lemma}
\begin{proof}
Suppose that $\#\{\Cal L^I: I\in \N_0^n\}=\infty.$ Then there exists an infinite sequence $\{\Cal L^{I_k}\}_{k=1}^{\infty}$ consisting of pairwise different ideals. We write
$I_k=(i_{k1},\dots,i_{kn}).$ Since $i_{k\ell}\in\N_0$, there exists an infinite sequence of positve integers $p_1<p_2<p_3<\cdots$ such that
$i_{p_1\ell}\leq i_{p_2\ell}\leq i_{p_3\ell}\leq\cdots$, for all $\ell=1,\dots,n$\,: In fact, firstly we choose  a sub-sequence $i_{q_11}\leq i_{q_21}\leq i_{q_31}\leq \cdots$  of $\{i_{k1}\}_{k=1}^{\infty}.$ Next, we choose a sub-sequence of $i_{r_12}\leq i_{r_22}\leq i_{r_32}\leq \cdots$ of $\{i_{q_k2}\}_{k=1}^{\infty}.$  Continuing the above process until obtaining a sub-sequence $i_{p_1n}\leq i_{p_2n}\leq i_{p_3n}\leq\cdots.$

We now prove that:
\begin{align}\label{Noertherian}
\Cal L^{I_{p_1}}\subset \Cal L^{I_{p_2}}\subset\Cal L^{I_{p_3}}\subset\cdots.
\end{align}
Indeed, for any  $\gamma\in\Cal L_N^{I_{p_k}}$ (for any $N$ and $k$ satisfying $N-\Vert I_{p_k}\Vert \geq 0$), we have
\begin{align*}
Q_1^{i_{p_k1}}\cdots Q_n^{i_{p_kn}}\gamma-
\sum_{E=(e_1,\dots,e_n)>I_{p_k}}Q_1^{e_1}\cdots Q_n^{e_n}\gamma_E\in \mathcal I_{\Cal K}(V)_N,
\end{align*}
 for some $\gamma_E\in \Cal K [x_0,\dots,x_M]_{N-d\Vert E\Vert}.$

\noindent Then, since $i_{p_{k+1}1}-i_{p_k1},\dots, i_{p_{k+1}n}-i_{p_kn}$ are non-negative integers, we have
\begin{align*}
Q_1^{i_{p_{k+1}1}}\cdots Q_n^{i_{p_{k+1}n}}\gamma-
\sum_{E=(e_1,\dots,e_n)>I_{p_k}}Q_1^{e_1+(i_{p_{k+1}1}-i_{p_k1})}\cdots Q_n^{e_n+(i_{p_{k+1}n}-i_{p_kn})}\gamma_E\in \mathcal I_{\Cal K}(V)_N.
\end{align*}
On the other hand since $E=(e_1,\dots,e_n)>I_{p_k}$ we have $(e_1+i_{p_{k+1}1}-i_{p_k1},\dots,e_n+i_{p_{k+1}n}-i_{p_kn})>I_{p_{k+1}}$. Therefore, $\gamma\in\Cal L_{N-d\Vert I_{p_k}\Vert+d\Vert I_{p_{k+1}}\Vert}^{I_{p_{k+1}}}.$
Hence, $\Cal L_N^{I_{p_k}}\subset \Cal L_{N-d\Vert I_{p_k}\Vert+d\Vert I_{p_{k+1}}\Vert}^{I_{p_{k+1}}}$ for all $k, N.$ Therefore, $\Cal L^{I_{p_k}}\subset \Cal L^{I_{p_{k+1}}}$ for all $k$. We get (\ref{Noertherian}).

\noindent Since $\Cal K[x_0,\dots,x_M]$ is a noetherian ring, the chain of ideals in (\ref{Noertherian}) becomes finally stationary. This is a contradiction.
\end{proof}

\begin{lemma}\label{newlm} There are   integers $n_0$, $c$  and $c'$ such that the following assertions hold.

i) \;$\dim_{\Cal K} \frac{\Cal K[x_0,\dots,x_M]_{N-d\Vert I\Vert}}{(\Cal I(V), Q_1,\dots,Q_n)_{N-d\Vert I\Vert}}=c$
for all $I\in\N_0^n, N\in\N_0$ satisfying $N-d\Vert I\Vert\geq n_0.$

ii) \;For each $I\in \N_0^n$ there is an integer $m^I$ such that $m^I=m_N^I$ for all $N\in\N_0$ satisfying $N-d\Vert I\Vert\geq n_0.$

iii) \; $m_N^I\leq c',$ for all $I\in\N_0^n$ and $N\in\N_0$ satisfying $N-d\cdot \Vert I\Vert\geq 0.$
\end{lemma}
\begin{proof}
 For each $z$ in $\C$ such that
all coefficients of $Q_j\;(j=1,\dots,n\}$ are holomorphic at $z,$  we denote  by $(\Cal I(V), Q(z),\dots, Q(z))$
the ideal in $\C[x_0,\dots,x_M]$ generated by $\Cal I(V)\cup\{Q_1(z),\dots,Q_n(z)\}.$

We have
\begin{align}\label{a7}
(\Cal I(V),Q_1(z),\dots,Q_n(z))\subset (\Cal I(V),Q_1,\dots,Q_n)(z).
\end{align}
Indeed,  for any $P\in(\Cal I(V),Q_1(z),\dots,Q_n(z)),$ we write $P=G+Q_1(z)\cdot P_1+\cdots+Q_n(z)\cdot P_n,$ where $G\in\Cal I (V),$
and $P_i\in\C[x_0,\dots,x_M].$ Take $\widetilde P:=G+Q_1\cdot P_1+\cdots+Q_n\cdot P_n\in (\Cal I(V), Q_1,\dots, Q_n),$ then all
coefficients of $\widetilde P$ are holomorphic at $z.$ It is clear that $\widetilde P(z)=P.$
Hence, we get (\ref{a7}).

\noindent Let $N$ be an arbitrary positive integer and $I$ be an arbitrary element in $\tau_N.$ Let $\{h_k:=\sum_{j=1}^n Q_j\cdot R_{jk}+\sum_{j=1}^{m_k} \gamma_{jk}\cdot g_{jk}\}_{k=1}^K$
 be a basic system of $(\Cal I(V), Q_1,\dots,Q_n)_{N-d\cdot \Vert I\Vert},$ where $g_{jk}\in \Cal I(V),$ and
 $R_{jk}, \gamma_{jk},\in \Cal K[x_0,\dots,x_M]$ satisfying $\deg (Q_j\cdot R_{jk})=\deg (\gamma_{jk}\cdot g_{jk})=N-d\cdot \Vert I\Vert.$
By Lemma \ref{L0}, and since $\{Q_0,\dots,Q_n\}$ is a $V-$ admissible set, there exists $a\in\C$ such that:

i) $\{h_k(a)\}_{k=1}^K$ is a basic system of $(\Cal I(V),Q_1,\dots,Q_n)_{N-d\cdot \Vert I\Vert}(a),$

ii)  all coefficients of
$Q_j, R_{jk}, \gamma_{jk}, g_{jk}$ are holomorphic at $a,$ and

iii) the homogeneous polynomials $Q_0(a),\dots,Q_n(a)\in \C[x_0,\dots,x_M]$ have no common zero points in $V.$

 \noindent On the other hand, it is clear that $h_k(a)\in(\Cal I(V),Q_1(a),\dots,Q_n(a)),$ for all $k=1,\dots,K.$
Hence, by (\ref{a7}), and by i), we have
\begin{align*}
(\Cal I(V),Q_1(a),\dots,Q_n(a))_{N-d\cdot \Vert I\Vert}= (\Cal I(V),Q_1,\dots,Q_n)_{N-d\cdot \Vert I\Vert}(a).
\end{align*}
Then, we have
\begin{align*}
\dim_{\Cal K}(\Cal I(V),Q_1,\dots,Q_n)_{N-d\cdot \Vert I\Vert}=K&=\dim_{\C}(\Cal I(V),Q_1,\dots,Q_n)_{N-d\cdot \Vert I\Vert}(a)\\
&=\dim_{\C}(\Cal I(V),Q_1(a),\dots,Q_n(a))_{N-d\cdot \Vert I\Vert}.
\end{align*}
Therefore,
\begin{align}\label{aa80}
\dim_{\Cal K}\frac{\Cal K[x_0,\dots,x_M]_{N-d\Vert I\Vert}}{(\Cal I(V),Q_1,\dots,Q_n)_{N-d\Vert I\Vert}}
= \dim_{\C}\frac{\C[x_0,\dots,x_M]_{N-d\Vert I\Vert}}{(\Cal I(V),Q_1(a),\dots,Q_n(a))_{N-d\Vert I\Vert}}.
\end{align}
On the other hand, by the Hilbert-Serre Theorem (\cite{H}, Theorem 7.5), there exist positive integers $n_1, c$  such that
\begin{align*}
 \dim_{\C}\frac{\C[x_0,\dots,x_M]_{N-d\Vert I\Vert}}{(\Cal I(V),Q_1(a),\dots,Q_n(a))_{N-d\Vert I\Vert}}=c,
\end{align*}
for all $I\in\N_0^n$ and $ N\in\N_0$ satisfying $N-d\Vert I\Vert\geq n_1.$

\noindent Combining with (\ref{aa80}), we have
\begin{align}\label{aa8}
\dim_{\Cal K}\frac{\Cal K[x_0,\dots,x_M]_{N-d\Vert I\Vert}}{(\Cal I(V),Q_1,\dots,Q_n)_{N-d\Vert I\Vert}}=c,
\end{align}
for all $I\in\N_0^n$ and $N\in\N_0$  satisfying $N-d\Vert I\Vert\geq n_1.$

Let $h^I$ and $h$ be the Hilbert functions of $\frac{\Cal K[x_0,\dots,x_M]}{\Cal L^I}$ and $\frac{\Cal K[x_0,\dots,x_M]}{(\Cal I(V), Q_1,\dots,Q_n)},$ respectively. Since  $(\Cal I(V), Q_1,\dots,Q_n)\subset \Cal L^I,$ we have $h^I\leq h.$ On the other hand, by Matsumura \cite{Ma}, Theorem 14,  $h^I(k)$ is a polynomial in $k$ for all $k>>0$ and by (2.12), we have $h(k)=c$ for all $k\geq n_1.$ Hence, there are constants $m^I$, $n_2$ such that $h^I(k)=m^I$ for all $k\geq n_2$ and then $m_N^I=h^I(N-d\Vert I\Vert)=m^I$  for all $N\in \N_0$ satisfying $N-d\Vert I\Vert\geq n_2.$ By Lemma~\ref{inftylemma}, we may choose $n_2$ common for all $I.$
Taking $n_0:=\max\{n_1,n_2\},$ we get Lemma 2.8, i) and ii).

We have $m^I_N=h^I(N-d\Vert I\Vert)\leq h(N-d\Vert I\Vert)\leq\max\{c, h(k): k=0,\dots, n_0\}.$ Hence, taking $c':=\max\{c, h(k): k=0,\dots, n_0\},$ we get  Lemma~2.8 iii).

\end{proof}

Set
\begin{align*} m:=\min_{I \in \N_0^n}m^I.
\end{align*}
We fix  $I_0=(i_{01},\dots,i_{0n})\in \N_0^n,$ and  $N_0\in \N_0$ such that $N_0-d\Vert I_0\Vert\geq n_0$ and $m^{I_0}_{N_0}=m.$

For each  positive integer $N,$  divisible by $d,$ denote by $\tau_N^0$ the set of all $I=(i_1,\dots,i_n)\in\tau_N$ such that $N-d\Vert I\Vert \geq n_0$ and $i_k\geq \max\{i_{01},\dots,i_{0n}\}$,  for all $k\in\{1,\dots, n\}.$

\noindent We have
\begin{align}\label{ph1} \#\tau_N=\binom{\frac{N}{d}+n}{n}&=\frac{1}{d^n}\cdot\frac{N^n}{n!}+O(N^{n-1}), \notag\\
                                        \#\{I\in\tau_N: N-d\Vert I\Vert\leq n_0\}&=O(N^{n-1}),\notag\\
                                         \#\{I=(i_1,\dots,i_n)\in\tau_N:\; i_k&<\max_{1\leq\ell\leq n}i_{0\ell}, \;\text{for some}\;k\}=O(N^{n-1}),\;\text{and\: so}\notag\\
                                         \#\tau_N^0&=\frac{1}{d^n}\cdot\frac{N^n}{n!}+O(N^{n-1}).
\end{align}
\begin{lemma} \label{L2}
  $m_N^I=\deg V\cdot d^n$ for all $N>>0,$  divisible by $d,$ and  $I\in \tau^0_N.$
\end{lemma}
\begin{proof}
 For any $\gamma\in\Cal L^{I^0}_{N_0},$ we have
\begin{align*}
T:=Q_1^{i_{01}}\cdots Q_n^{i_{0n}}\gamma-\sum\limits_{E=(e_{1},\dots,e_{n})>I_0}Q_1^{e_{1}}\cdots Q_n^{e_{n}}\gamma_E\in\Cal I_{\Cal K}(V)_{N_0},
\end{align*}
for some $\gamma_E\in\Cal K[x_0,\dots,x_M]_{N-d\Vert E\Vert}.$

\noindent Then, for any $I=(i_1,\dots,i_n)\in\tau_N^0,$ we have
\begin{align}\label{h3}
Q_1^{i_{1}}\cdots Q_n^{i_{n}}\gamma&-\sum\limits_{E=(e_{1},\dots,e_{n})>I_0}Q_1^{e_{1}+i_1-i_{01}}\cdots Q_n^{e_{n}+i_n-i_{0n}}\gamma_E\notag\\
&=Q_1^{i_{1}-i_{01}}\cdots Q_n^{i_{n}-i_{0n}}\cdot T\in\Cal I_{\Cal K}(V)_{N_0}.
\end{align}
On the other hand since $I\in\tau_N^0$ and $E>I_0,$ we have $(e_{1}+i_1-i_{01},\dots,e_{n}+i_n-i_{0n})>I.$

\noindent Hence, by (\ref{h3}) we have
$$\gamma\in\Cal L^I_{N_0+d\Vert I\Vert-d\Vert I_0\Vert}.$$
This implies that
\begin{align*}
\Cal L^{I_0}_{N_0}\subset \Cal L^I_{N_0+d\Vert I\Vert-d\Vert I_0\Vert}.
\end{align*}
Then
\begin{align}
m=m^{I_0}_{N_0}&=\dim_{\Cal K}\frac{\Cal K[x_0,\dots,x_M]_{N_0-d\Vert I_0\Vert}}{\Cal L^{I_0}_{N_0}}\notag \\
&\geq\dim_{\Cal K}\frac{\Cal K[x_0,\dots,x_M]_{N_0-d\Vert I_0\Vert}}{\Cal L^I_{N_0+d\Vert I\Vert-d\Vert I_0\Vert}}\notag \label{tt}\\
&=m^I_{N_0+d\Vert I\Vert-d\Vert I_0\Vert}.
\end{align}
On the other hand since $\left(N_0+d\Vert I\Vert-d\Vert I_0\Vert\right)-d\Vert I\Vert =N_0-d\Vert I_0\Vert\geq n_0,$ and  $N-\Vert I\Vert\geq n_0$  (note that $I\in\tau_N^0$), by Lemma \ref{newlm}, we have
 $$m^I=m^I_{N_0+d\Vert I\Vert-d\Vert I_0\Vert}=m^I_{N}.$$
Hence, by (\ref{tt}), $m\geq m^I=m^I_{N}.$ Then, by the minimum property of $m,$ we get that
\begin{align}\label{np1}
m^I_N=m\; \text{for all}\; I\in\tau_N^0.
\end{align}

We now prove that:
\begin{align}\label{cl} \dim_{\Cal K}\Cal I_{\Cal K}(V)_N=\dim_{\C}\Cal I(V)_N.
\end{align}
Indeed,  let $\{P_1,\dots, P_s\}$ be a basis of the $\C-$ vector space $\Cal I(V)_N.$
It is clear that $\Cal I_{\Cal K}(V)_N$ is a vector space over $\Cal K$ generated by $\Cal I(V)_N,$ therefore $\{P_1,\dots, P_s\}$ is
also a generating system of $\Cal I_{\Cal K}(V)_N.$ Then,  for (\ref{cl}), it suffices to prove that
if $t_1,\dots, t_s\in\Cal K$ satisfy
\begin{align}\label{a6}
t_1\cdot P_1+\cdots+ t_s\cdot P_s\equiv 0,
\end{align}
then $t_1=\dots=t_s\equiv 0.$
We rewrite (\ref{a6}) in the following form
\begin{align*}
A\cdot \left ( \begin{matrix}
t_1\cr
\cdot\cr
\cdot\cr
\cdot\cr
t_s\end{matrix}\right)=\left ( \begin{matrix}
0\cr
\cdot\cr
\cdot\cr
\cdot\cr
0\end{matrix}\right),
\end{align*}
where $A\in$Mat$(\binom{M+N}{N}\times s,\Cal K).$

\noindent If the above system of linear equations has non-trivial solutions, then rank$_{\Cal K} A<s.$ Then
rank$_{\C} A(z)<s$ for all $z\in\C$ excluding a discrete set. Take $a\in\C$ such that
rank$_{\C} A(a)<s.$ Then the following system of linear equations
\begin{align*}
A(a)\cdot \left ( \begin{matrix}
t_1\cr
\cdot\cr
\cdot\cr
\cdot\cr
t_s\end{matrix}\right)=\left ( \begin{matrix}
0\cr
\cdot\cr
\cdot\cr
\cdot\cr
0\end{matrix}\right),
\end{align*}
has some non-trivial solution $(t_1,\dots,t_s)=(\alpha_1,\dots,\alpha_s)\in\C^s\setminus\{0\}.$
Then $\alpha_1\cdot P_1+\cdots+\alpha_s\cdot P_s\equiv 0,$ this is
a contradiction. Hence, we get (\ref{cl}).

By Lemma \ref{L1} and (\ref{cl}), we have
\begin{align}
\sum_{I\in\tau_N}m_N^I=\dim_{\Cal K}\frac{\Cal K[x_0,\dots,x_M]_N}{\Cal I_{\Cal K}(V)_N}
&=\dim_{\C}\frac{\C[x_0,\dots,x_M]_N}{\Cal I(V)_N}\notag\\
&=\deg V\cdot \frac{N^n}{n!} +O(N^{n-1}), \label{ttt}
\end{align}
for all $N$ large enough. \\
Combining with (\ref{np1}), we have
\begin{align}\label{aa9}
m\cdot\#\tau_N^0+\sum_{I\in\tau_N\setminus\tau_N^0}m_N^I=\deg V\cdot \frac{N^n}{n!} +O(N^{n-1}).
\end{align}
On the other hand by Lemma \ref{newlm}, $m_N^I\leq c',$ for all $I\in\tau_N\setminus\tau_N^0.$ Hence, by (\ref{ph1}), we have
\begin{align*}
m=\deg V\cdot d^n.
\end{align*}
Combining with (\ref{np1}), we have
\begin{align*}
m^I_N=\deg V\cdot d^n
\end{align*}
for all $I\in\tau_N^0.$
\end{proof}
\begin{lemma}\label{L3}
For each $s\in\{1,\dots,n\},$ and for $N>>0,$ divisible by $d,$ we have:
\begin{align*}
\sum_{I=(i_1,\dots,i_n)\in\tau_N}m_N^I\cdot i_s\geq \frac{\deg V}{d\cdot (n+1)!}N^{n+1}-O(N^n).
\end{align*}
\end{lemma}
\begin{proof}
Firstly, we note that if $I=(i_1,\dots,i_n)\in\tau^0_N,$ then all symmetry $I'=(i_{\sigma(1)},\dots,i_{\sigma(n)})$ of $I$ also belongs to $\tau^0_N.$ On the other hand, by
 Lemma~\ref{L2}, we have $m_N^I=\deg V\cdot d^n,$ for all $I\in\tau_N^0.$ Therefore, by (\ref{ph1}) we have
\begin{align*}
\sum_{I=(i_1,\dots,i_n)\in\tau_N^0}&m_N^I\cdot i_1=\cdots=\sum_{I=(i_1,\dots,i_n)\in\tau_N^0}m_N^I\cdot i_n\notag\\
&=\deg V\cdot d^n\cdot \sum_{I\in\tau_N^0}\frac{\Vert I\Vert}{n}\notag\\
&=\deg V\cdot d^n\cdot \left(\sum_{I\in\tau_N}\frac{\Vert I\Vert}{n}-
 \sum_{I\in\tau_N\setminus\tau^0_N}\frac{\Vert I\Vert}{n}\right)\notag\\
&\geq\deg V\cdot d^n\left(\sum_{k=0}^{\frac{N}{d}}\frac{k}{n}
\cdot\binom{k+n-1}{n-1}-\left(\#\tau_N-\#\tau_N^0\right)
\cdot \frac{N}{nd}\right)\notag\\
&=\deg V\cdot d^n\left(\sum_{k=0}^{\frac{N}{d}}\frac{k}{n}
\cdot\binom{k+n-1}{n-1}-O(N^{n-1})
\cdot \frac{N}{nd}\right)\notag
\end{align*}
\begin{align*}
&=\deg V\cdot d^n\sum_{k=1}^{\frac{N}{d}}\binom{k+n-1}{n}-O(N^{n})\notag\\
&=\deg V \cdot d^n \binom{\frac{N}{d}+n}{n+1}-O(N^{n})\notag\\
&\geq \frac{\deg V}{d\cdot (n+1)!}N^{n+1}-O(N^n).
\end{align*}
Hence, for each $i\in\{1,\dots,n\}$
\begin{align*}
\sum_{I=(i_1,\dots,i_n)\in\tau_N}m_N^I\cdot i_s&\geq \sum_{I=(i_1,\dots,i_n)\in\tau_N^0}m_N^I\cdot i_s\notag
\\
&\geq \frac{\deg V}{d\cdot (n+1)!}N^{n+1}-O(N^n).
\end{align*}
\end{proof}
 We recall that by (\ref{ttt}), for $N>>0,$ we have
$$\dim_{\cal K}\frac{\Cal K[x_0,\dots,x_M]_N}{\Cal I_{\Cal K}(V)_N}
=H_V(N)
=\deg V\cdot \frac{N^n}{n!} +O(N^{n-1}).$$ Therefore, from Lemmas \ref{L1}, \ref{L3} we get immediately the following result.
\begin{lemma}\label{L4}
For all $N>>0$ divisible by $d,$ there are homogeneous polynomials $\phi_1,\dots,\phi_{H_V(N)}$ in $\Cal K[x_0,\dots,x_M]_N$ such that they form a basis
 of the $\mathcal K-$ vector space
$\frac{\mathcal K_[x_0,\dots,x_M]_N}{\mathcal I_{\Cal K}(V)_N},$  and
 \begin{align*}
\prod_{j=1}^{H_V(N)}\phi_j-\big(Q_{1}\cdots Q_{n}\big)^{\frac{\deg V\cdot N^{n+1}}{d\cdot (n+1)!}-u(N)}\cdot P\in\mathcal I_{\Cal K}(V)_{N \cdot H_V(N)}\,,
\end{align*}
where  $u(N)$ is a function in $N$ satisfying $u(N)\leq O(N^n),$ $P$ is a homogeneous polynomial of degree
$$N\cdot H_V(N)-\frac{n\cdot\deg V\cdot N^{n+1}}{(n+1)!}+ nd\cdot u(N)=\frac{\deg V\cdot N^{n+1}}{(n+1)!}+O(N^n).$$
\end{lemma}
\begin{lemma}[see \cite{R1}]\label{L5} Let $f$ be a non-constant holomorphic map of $\C$ into $\C \P^M.$ Let $H_j=a_{j0}x_0+\cdots+a_{jM}x_M, j\in\{1,\dots,q\}$ be
 $q$ linear homogeneous polynomials in $\Cal K_f[x_0,\dots,x_M].$   Denote by $\Cal K_{\{H_j\}_{j=1}^q}$ the field over $\C$ of  all
meromorphic functions on $\C$ generated by $\{a_{ji}, i=0,\dots,M; j=1,\dots,q\}.$ Assume that $f$ is linearly non-degenerate over $\Cal K_{\{H_j\}_{j=1}^q}.$
Then for
each $\epsilon>0,$ we have
\begin{align}\label{inequi}
\Big\Vert \frac{1}{2\pi}\int_0^{2\pi}\max_{K}\log\prod_{k\in K}\Big(\frac{\Vert f\Vert\cdot\max\limits_{i=0,\dots,M}|a_{ki}|}{|H_k(f)|}(re^{i\theta})\Big)d\theta\leq (M+1+\epsilon)T_f(r),
\end{align}
where $\max_K$ is taken over all subsets $K\subset\{1,\dots,q\}$ such that the polynomials $ H_j,$ $ j\in K$ are linearly independent over $\Cal K_{\{H_j\}_{j=1}^q}.$
\end{lemma}
\begin{remark}\label{R2} Since the coefficients of the $H_j'$s are small functions (with respect to $f$), by  the First Main Theorem, and by (\ref{inequi}), for
each $\epsilon>0,$ we have
\begin{align*}
\Big\Vert \frac{1}{2\pi}\int_0^{2\pi}\max_{K}\log\prod_{k\in K}\Big(\frac{\Vert f\Vert}{|H_k(f)|}(re^{i\theta})\Big)d\theta\leq (M+1+\epsilon)T_f(r).
\end{align*}
\end{remark}

\section{Proof of the Main Theorem}
Replacing $Q_j$ by $Q^{\frac{d}{d_j}}$, where $d$ is the l.c.m of the $Q_j$'s, we may assume that the polynomials $Q_1,\dots,Q_q$ have the same degree $d.$
Let $N>>0$ be an integer divisible by $d.$
For each $J:=\{j_1,\dots,j_n\}\subset\{1,\dots,q\}, $ by Lemma \ref{L4} (for $\Cal K:=\Cal K_{\Cal Q}$), there are homogeneous polynomials $\phi_1^J,\dots,\phi_{H_V(N)}^J$
(depending on $J$) in $\Cal K_{\Cal Q}[x_0,\dots,x_M]$ and there are functions (common for all $J$)  $u(N), v(N)\leq O(N^n)$ such that they form  a basis
 of the $\mathcal K_{\Cal Q}-$ vector space
$\frac{\mathcal K_{\Cal Q}[x_0,\dots,x_M]}{\mathcal I_{\Cal K_{\Cal Q}}(V)_N},$  and
 \begin{align*}
\prod_{\ell=1}^{H_V(N)}\phi_\ell^J-\big(Q_{j_1}\cdots Q_{j_n}\big)^{\frac{\deg V\cdot N^{n+1}}{d\cdot (n+1)!}-u(N)}\cdot P_J\in\mathcal I_{\Cal K_{\Cal Q}}(V)_{N \cdot H_V(N)}\,,
\end{align*}
where $P_J$ is a homogeneous polynomial of degree
$\frac{\deg V\cdot N^{n+1}}{(n+1)!}+v(N).$

\noindent On the other hand, for any $Q\in\mathcal I_{\Cal K_{\Cal Q}}(V)_{N \cdot H_V(N)}$, we have $Q(f)\equiv 0.$

\noindent Therefore
 \begin{align*}
\prod_{\ell=1}^{H_V(N)}\phi_\ell^J(f)=\big(Q_{j_1}(f)\cdots Q_{j_n}(f)\big)^{\frac{\deg V\cdot N^{n+1}}{d\cdot (n+1)!}-u(N)}\cdot P_J(f).
\end{align*}
Since the coefficients of $P_J$ are small functions (with respect to $f$), it is easy to  see that there exist $h_J\in\Cal C_f$ such that
\begin{align*}
|P_J(f)|\leq \Vert f\Vert ^{\deg P_J}\cdot h_J=\Vert f\Vert ^{\frac{\deg V\cdot N^{n+1}}{(n+1)!}+v(N)}\cdot h_J.
\end{align*}
Hence,
 \begin{align*}
\log\Big(\prod_{\ell=1}^{H_V(N)}|\phi_\ell^J(f)|\Big)&
\leq\Big(\frac{\deg V\cdot N^{n+1}}{d\cdot (n+1)!}-u(N)\Big)\cdot\log\big|Q_{j_1}(f)\cdots Q_{j_n}(f)\big|+\log^+ h_J\\
&\quad\quad+
\big(\frac{\deg V\cdot N^{n+1}}{(n+1)!}+v(N)\big)\cdot\log\Vert f\Vert.
\end{align*}
This implies that there are functions $\omega_1(N), \omega_2(N)\leq O(\frac{1}{N})$ such that
\begin{align}\label{p1}
\log\big(|Q_{j_1}(f)|\cdots |Q_{j_n}(f)|\big)
\geq\Big(\frac{d\cdot (n+1)!}{\deg V\cdot N^{n+1}}&
-\frac{\omega_1(N)}{N^{n+1}}\Big)
\cdot \log\Big(\prod_{\ell=1}^{H_V(N)}|\phi_\ell^J(f)|\Big)
\notag\\-
\frac{1}{N^{n+1}}\log^+ \widetilde h_J
&-(d+\omega_2(N)\big)\cdot\log\Vert f\Vert,
\end{align}
for some $\widetilde h_J\in \Cal C_f.$

We fix homogeneous polynomials $\Phi_1,\dots,\Phi_{H_V(N)}\in \mathcal K_{\Cal Q}[x_0,\dots,x_M]_N$ such that they form a basis of the $\mathcal K_{\Cal Q}-$ vector space
$\frac{\mathcal K_{\Cal Q}[x_0,\dots,x_M]_N}{\mathcal I_{\Cal K_{\Cal Q}}(V)_N}.$ Then for each subset $J:=\{j_1,\dots,j_n\}\in\Cal \{1,\dots,q\},$
there exist homogeneous linear polynomials $L_1^J,\dots L_{H_V(N)}^J\in\mathcal K_{\Cal Q}[y_1,\dots,y_{H_V(N)}]$ such that they are linearly independent
over $\mathcal K_{\Cal Q}$ and

\begin{align}\label{p2}
\phi_\ell^J-L^J_\ell(\Phi_1,\dots,\Phi_{H_V(N)})\in\mathcal I_{\Cal K_{\Cal Q}}(V)_N,\;\text{for all}\;\ell\in\{1,\dots,H_V(N)\}.
\end{align}

It is easy to see that there exists a meromorphic function
$\varphi$ such that
 $N_{\varphi}(r)=o(T_f(r)),$  $N_{\frac{1}{\varphi}}(r)=o(T_f(r))$ and $\frac{\Phi_1(f)}{\varphi},\dots,\frac{\Phi_{H_V(N)}(f)}{\varphi}$ are
holomorphic and have no common zeros (note that all coefficients of $\Phi_\ell$ are in $\Cal K_{\Cal Q}\subset\Cal K_f$).

\noindent Let $F:\C \to\C \P^{H_V(N)-1}$ be the holomorphic map with the reduced representation
$F:=\big(\frac{\Phi_1(f)}{\varphi}:\cdots:\frac{\Phi_{H_V(N)}(f)}{\varphi}\big).$ Since $f$ is algebraically nondegenerate over
 $\Cal K_{\Cal Q},$ and since the polynomials $\Phi_1,\dots,\Phi_{H_V(N)}$ form a basis of
$\frac{\mathcal K_{\Cal Q}[x_0,\dots,x_M]_N}{\mathcal I_{\Cal K_{\Cal Q}}(V)_N}$, we get that $F$ is linearly non-degenerate over
$\Cal K_{\Cal Q}.$ As a corollary, $F$ is  linearly non-degenerate over the field over $\C$ generated by all coefficients of $L_\ell$'s.

It is easy to see that
\begin{align}\label{p3}
T_F(r)\leq N\cdot T_f(r)+o(T_f(r)).
\end{align}

In order to simplify the writing of the following series of inequalities,
put $A(N):=\Big(\frac{d\cdot (n+1)!}{\deg V\cdot N^{n+1}}
-\frac{\omega_1(N)}{N^{n+1}}\Big)$.
By (\ref{p2}), for all $\ell\in\{1,\dots,H_V(N)\}$ we have
\begin{align*}
\log|\phi_\ell^J(f)|=\log|L_\ell^J(F)|+\log|\varphi|.
\end{align*}
Hence, by (\ref{p1}),
 and by taking $\widetilde h\in\Cal C_f$ such that $\log^+ \widetilde h_J \leq \log^+ \widetilde h$ for all $J$,
we get
\begin{align}\label{p4}
\log\big(|Q_{j_1}(f)|\cdots |Q_{j_n}(f)|\big)&\geq
A(N) \cdot
\Big(H_V(N) \cdot \log | \varphi | +
\log\Big(\prod_{\ell=1}^{H_V(N)}|L_\ell^J(F)|\Big)\Big)\notag\\
&\quad\quad-\frac{1}{N^{n+1}}\log^+ \widetilde h_J-\big(d+\omega_2(N)\big)\log\Vert f\Vert\notag\\
&\geq
A(N) \cdot
\log\Big(\prod_{\ell=1}^{H_V(N)}|L_\ell^J(F)|\Big) + A(N) \cdot  H_V(N) \cdot \log | \varphi | \notag\\
&\quad\quad-\log^+ \widetilde h-\big(d+\omega_2(N)\big)\log\Vert f\Vert\,.
\end{align}
Then, by Lemma \ref{inertia},  and by increasing $\widetilde h\in\Cal C_f$ if necessary, we get
\begin{align}\label{3.5}
\text{log}\prod_{j=1}^q |Q_j(f)| &=
\max_{\{\beta_1,\dots,\beta_{q-n}\} \subset \{1,\dots,q\}}
\text{log}|Q_{\beta_1}(f) \cdots Q_{\beta_{q-n}}(f)| \notag \\
&\qquad + \min_{J = \{j_1,\dots,j_n\} \subset \{1,\dots,q\}}
\text{log}|Q_{j_1}(f) \cdots Q_{j_n}(f)| \notag \\
&\geq (q-n)d \cdot \text{log}\Vert f\Vert + \min_{J \subset\{1,\dots,q\},\#J=n} A(N)\cdot \log\Big(\prod_{\ell=1}^{H_V(N)}|L_\ell^J(F)|\Big) \notag \\
&\qquad - \big(d+\omega_2(N)\big)\log\Vert f\Vert + A(N) \cdot  H_V(N) \cdot  \log|\varphi|
 - \text{log}^+\widetilde h \notag \\
&= (q-n-1)d \cdot \text{log}\Vert f \Vert +\min_{J \subset
\{1,\dots,q\},\#J=n} A(N)\cdot \log\Big(\prod_{\ell=1}^{H_V(N)}|L_\ell^J(F)|\Big) \notag \\
&\qquad - \omega_2(N) \cdot \text{log}\Vert f \Vert
+ A(N) \cdot  H_V(N) \cdot  \log|\varphi| - \text{log}^+\widetilde h
\end{align}
Now for given $\epsilon >0$ we fix $N=N(\epsilon)$ big enough such that
\begin{equation}\label{3.6}
\omega_2(N)\leq \frac{\epsilon}{3}\: {\rm and}\: A(N) < 1\,.
\end{equation}
By using Remark \ref{R2} to the holomorphic map
$F:\C \to\C \P^{H_V(N)-1}$, the error constant $\frac{\epsilon}{2N}>0$ and
the system of linear polynomials $L_1^J,\dots L_{H_V(N)}^J\in\mathcal K_{\Cal Q}[y_1,\dots,y_{H_V(N)}]$, where $J$ runs over all subsets  $J:=\{j_1,\dots,j_n\}\in\Cal \{1,\dots,q\}$,
we get:
\begin{align} \label{3.7}
\Big\Vert\frac{1}{2\pi}\int_{0}^{2\pi}\max_{J \subset
\{1,\dots,q\},\# J=n}\log\Big(\prod_{\ell=1}^{H_V(N)}\frac{\Vert F\Vert }{|L_\ell^J(F)|}(re^{i\theta})\Big) d\theta \notag \\
\leq
 \frac{1}{2\pi}\int_0^{2\pi}\max_{K}\log\prod_{k\in K}\Big(\frac{\Vert F\Vert}{|L_\ell^J(F)|}(re^{i\theta})\Big)d\theta\leq (H_V(N)+\frac{\epsilon}{2N})T_F(r)\,,
\end{align}
where $\max_K$ is taken over all subsets of
the system of linear polynomials $L_1^J,\dots L_{H_V(N)}^J\in\mathcal K_{\Cal Q}[y_1,\dots,y_{H_V(N)}]$, where $J$ runs over all subsets  $J:=\{j_1,\dots,j_n\}\in\Cal \{1,\dots,q\}$,
such that these linear polynomials  are linearly independent over $\Cal K_{\Cal Q}$.

 So, by integrating (\ref{3.5}) and combining with (\ref{3.6}) and (\ref{3.7}) we have
 (using that $N_{\varphi}(r)=o(T_f(r)),$  $N_{\frac{1}{\varphi}}(r)=o(T_f(r))$,
 that $A(N) \cdot H_V(N) \leq O( \frac{1}{N})$
 and that  $\widetilde h\in\Cal C_f$)
\begin{align*}
\Big\Vert \sum_{j=1}^qN_f(r,Q_j)\geq &d(q-n-1)T_f(r)- \frac{\epsilon}{3}T_f(r)
+ A(N) \cdot H_V(N) \cdot \Big(N_{\varphi}(r) - N_{\frac{1}{\varphi}}(r)  \Big)
-\frac{\epsilon}{12}T_f(r)\notag\\
&+A(N)\cdot\frac{1}{2\pi}\int_{0}^{2\pi}\min_{J \subset
\{1,\dots,q\},\# J=n}\log\Big(\prod_{\ell=1}^{H_V(N)}{|L_\ell^J(F)|}(re^{i\theta})\Big) d\theta \notag\\
\geq&d(q-n-1)T_f(r)- \frac{\epsilon}{3}T_f(r)- \frac{\epsilon}{12}T_f(r)-\frac{\epsilon}{12}T_f(r)\notag\\
&+A(N)\cdot\frac{1}{2\pi}\int_{0}^{2\pi}\min_{J \subset
\{1,\dots,q\},\# J=n}\log\Big(\prod_{\ell=1}^{H_V(N)}{|L_\ell^J(F)|}(re^{i\theta})\Big) d\theta \notag
\end{align*}
\begin{align*}
=&d(q-n-1)T_f(r)- \frac{\epsilon}{2}T_f(r)\notag\\
&-A(N)\cdot\frac{1}{2\pi}\int_{0}^{2\pi}\max_{J \subset
\{1,\dots,q\},\# J=n}\log\Big(\prod_{\ell=1}^{H_V(N)}\frac{\Vert F\Vert }{|L_\ell^J(F)|}(re^{i\theta})\Big) d\theta \notag\\
&+A(N) \cdot H_V(N)\cdot T_F(r)\notag\\
\geq &d(q-n-1)T_f(r)
-A(N)\big(H_V(N)+\frac{\epsilon}{2N}\big)T_F(r)\notag\\
&+ A(N) \cdot H_V(N) \cdot T_F(r)-\frac{\epsilon}{2}T_f(r)\notag\\
\geq&d(q-n-1)T_f(r)-\frac{\epsilon}{2N} T_F(r)-\frac{\epsilon}{2}T_f(r)\notag\\
\geq& d(q-n-1-\epsilon)T_f(r).
\end{align*}
This completes the proof of the Main Theorem.
\hfill$\square$

 \noindent Gerd Dethloff \\
Universit\'{e} de Bretagne Occidentale (Brest)\\
Laboratoire de Math\'{e}matiques de Bretagne Atlantique - UMR 6205\\
6, avenue Le Gorgeu, CS 93837 \\
 29238 Brest Cedex 3, France \\
e-mail: gerd.dethloff@univ-brest.fr\\

\noindent Tran Van Tan\\
Department of Mathematics\\
  Hanoi National University of Education\\
 136-Xuan Thuy street, Cau Giay, Hanoi, Vietnam\\
e-mail: tantv@hnue.edu.vn

\end{document}